\documentclass[12pt]{article}
\usepackage{amsmath,amssymb,amsfonts,amsthm,enumerate}
\usepackage{eucal} 
\usepackage{xcolor}

\newcommand{\cF}{\mathcal{F}}
\newcommand{\C}{\mathbb{C}}

\newcommand{\Z}{\mathbb{Z}}
\newcommand{\cO}{\mathcal{O}}
\newcommand{\cV}{\mathcal{V}}

\newcommand{\cN}{\mathcal{N}}

\newcommand{\Pp}{\mathbb{P}^1}

\newcommand{\cE}{\mathcal{E}}

\newcommand{\R}{\mathbb{R}}
\newcommand{\cH}{\mathcal{H}}

\newcommand{\cM}{\mathcal{M}}

\newcommand{\eps}{\epsilon}
\newcommand{\cA}{\mathcal{A}}

\newcommand{\cX}{\mathcal{X}}
\newcommand{\cU}{\mathcal{U}}
\newcommand{\cZ}{\mathcal{Z}}

\newcommand{\cY}{\mathcal{Y}}

\newcommand{\la}{\lambda}

\newcommand{\fl}{\mathcal{P}}
\newcommand{\somefield}{{\mathbf k}}

\DeclareMathOperator{\Ad}{Ad}

\DeclareMathOperator{\Hilb}{Hilb}
\DeclareMathOperator{\Ext}{Ext}

\DeclareMathOperator{\ch}{ch}

\DeclareMathOperator{\Res}{Res}
\DeclareMathOperator{\codim}{codim}
\DeclareMathOperator{\Gr}{Gr}

\newtheorem*{thma}{Theorem A}
\newtheorem{thm}{Theorem}

\newtheorem{lemma}{Lemma}

\begin{document}



\title{A projection formula for the ind-Grassmannian}
\author{Erik Carlsson}

\maketitle

\abstract{
Let $X = \bigcup_k X_k$ be the ind-Grassmannian of codimension $n$
subspaces of an infinite-dimensional torus representation.
If $\cE$ is a bundle on $X$, we expect that
$\sum_j (-1)^j \Lambda^j(\cE)$
represents the $K$-theoretic fundamental class $[\cO_Y]$ 
of a subvariety $Y \subset X$ dual to $\cE^*$.
It is desirable to lift a $K$-theoretic 
``projection formula'' from the finite-dimensional 
subvarieties $X_k$, but such a statement requires
switching the order of the limits in $j$ and $k$.
We find conditions in which this may be done,
and consider examples in which
$Y$ is the Hilbert scheme of points in the plane,
the Hilbert scheme of an irreducible curve singularity,
and the affine Grassmannian of $SL(2,\C)$. In the last example, 
the projection formula becomes an instance 
of the Weyl-Ka\c{c} character formula,
which has long been recognized as the result of formally extending Borel-Weil
theory and localization to $Y$ \cite{S}. See also \cite{C3} for a proof of
the MacDonald inner product formula of type $A_n$ along these lines.
}

\section{Introduction}

Let $X$ be a smooth complex projective variety, let
\[T \circlearrowright X,\quad T = \C^* =\{z\},\]
be a one-dimensional complex torus action on $X$,
and let $\cE$ be an equivariant bundle on $X$.
The $K$-theoretic Atiyah-Bott-Lefschetz localization formula 
describes the character of the derived push forward to a point, 
also known as the equivariant Euler characteristic,
\begin{equation}
\label{kloc}
\chi_X (\cE) = \sum_i (-1)^i\ch  H^i_X(\cE) =\sum_{F \subset X^T} 
\chi_F\left(\cE_F \lambda(\cN^*_{X/F})^{-1}\right),
\end{equation}
Here $H^i$ is the \c{C}ech cohomology group, 
$\ch$ is the (Chern) character map, $F$ ranges over the fixed
components of the torus action, $\cN_{X/F}$ is its normal bundle, 
and $\lambda$ is the usual operation on $K$-theory
%
%
defined below.
See \cite{CG} for a reference.

Suppose $Y \subset X$ is an invariant subvariety 
which is the zero set of an equivariant section of a bundle $\cE^*$.
Then the fundamental class $[\cO_Y]$ is given by 
$\lambda(\cE) \in K_T(X)$, and we have the projection formula
\begin{equation}
\chi_Y\left( \gamma \right) = \chi_X\left(\gamma \lambda(\cE)\right).
\label{proj}
\end{equation}
If we apply the localization formula to either side, the resulting identity
is not mysterious. The fundamental class $\lambda(\cE)$ 
vanishes when restricted to a component $F \subset X$ that does not 
intersect with $Y$, and the two expressions are in fact equal termwise.
However, if a formula is known describing $\chi_X$ as a Laurent polynomial
rather than an unworkable rational function, then \eqref{proj}
produces such a formula for $Y$.


We will derive a version of \eqref{proj}
for several examples in
which $Y$ is an interesting moduli space, and $X$ is
the Grassmannian of codimension $n$ subspaces of an infinite-dimensional
torus representation $Z$, defined as an ind-variety, i.e. a union
of finite dimensional subvarieties,
\[\cdots \subset X_{-1} \subset X_0 \subset X_1 \subset \cdots
X,\quad \bigcup_k X_k = X.\]
%
We define the $K$-theory of this space as the inverse limit
\[\tilde{K}_T(X) = \lim_{\leftarrow} K_T(X_k) \otimes \C[[z]],\quad
\chi_X(\gamma) = \lim_{k\rightarrow \infty} \chi_X(\gamma_k) \in \C((z)).\]
%
For instance, we have a class $[\cU]\in K_T(X)$, where
$\cU_k$ is the tautological rank $n$ quotient bundle on $X_k$.

Consider the fairly general virtual bundle
\[\cE = A \cU+B\cU^*+C\cU \cU^* \in \tilde{K}_T(X),\]
where $A,B,C$ are torus characters, i.e. elements of the
equivariant $K$-theory of a point.
%
%
We then define the class $\cY=\la(\cE)$ by its components
\begin{equation}
\label{cYkint}
\cY_k = \lim_{w\rightarrow 1} 
\sum_{j\geq 0} (-w)^j \la^j(\cE)\in \tilde{K}_T(X_k),
\end{equation}
and define $\chi_{\cY}(\gamma)=\chi_{X}(\gamma \cY)$.
We will think of $\cY$ as the fundamental class of some subvariety 
$Y \subset X$, if formally we have $\chi_{\cY}(\gamma)=\chi_{Y}(i^* \gamma)$, 
even though $\cE$ may not be
an honest bundle, and $Y$ may be noncompact, 
infinite dimensional, or singular.

Our main theorem is the following:
\begin{thma}
\label{mainthmint}
Under certain conditions on $\cE,X,\gamma$, we have the following
analog of \eqref{proj}:
%
\[\chi_\cY\left(\gamma \right) = 
\sum_{j\geq 0}(-1)^j\chi_{X} \left(\gamma \lambda^j (\cE)\right).\]

%

\end{thma}
\noindent
Essentially, this theorem gives conditions
under which we may switch the limits in $j$ and $k$ in \eqref{cYkint},
which turns out to imply the formula.
The essential part of the argument 
is to calculate the rational function in $\C(z,w)$
whose expansion in $w$ is the contribution of the higher
Cech cohomology groups to \eqref{cYkint},
and show that it vanishes to high degree at $z=0$.

We include the following examples:

\begin{enumerate}

\item \label{hilbexample} 
$Y$ is the Hilbert scheme of $n$ points in the complex plane,
and $X=G_{n,R}$, where $R$ is the total space of $\C[x,y]$. 
The imbedding is the map
which associates to a subscheme of $\C^2$ the total space its ideal. 
The projection formula becomes a power series
expansion with integer coefficients
for the Euler characteristic of a subbundle of $\cU^{\otimes m}$,
where $\cU$ is the tautological rank $n$ bundle on $Y$.
%
%
This example may also be extended
to the moduli space $\cM_{r,n}$ of higher rank sheaves (instantons), 
see \cite{Nak2} for a definition.

\item $Y$ is the Hilbert scheme of a plane curve singularity
$y^2=x^3$, and $X=G_{n,R}$ with
%
\[R=\C[x,y]/(y^2-x^3) \cong \C[u^2,u^3].\]
Again the projection formula produces a power series formula
form the euler characteristic. Since $Y$ is singular, $\chi_Y$
must be defined using virtual localization.

\item $Y$ is the affine Grassmannian of the loop group
of $SL(2,\C)$, and $X$ is the Sato
Grassmannian of half infinte-dimensional subspaces of a faithful
representation of $L\C^2$. The projection formula produces
an instance of the Weyl-Ka\c{c} character formula.
This circumvents the technical ideas behind an idea that 
was discussed by Segal in \cite{S}, and has been studied
by several authors, including generalizations to the analogous flag varieties 
\cite{Ku,T1,T2}.

\end{enumerate}
%
In \cite{C3}, the author also gives a proof of the MacDonald inner product
formula of type $A_n$ in this way, which is too involved to reproduce here. 
From this point of view, the factorization of the inner product is explained 
as coming from a localization sum concentrated at a single fixed point point,
together with the form of the Pieri rules for MacDonald polynomials.

The motivations for this paper have do to with a fascinating and well studied
interplay between the Hilbert scheme of points on a surface,
representation theory, and modular forms. In many different
studies, geometric correspondences between the Hilbert schemes 
of different points induce an action of various infinite-dimensional
Lie algebras on $\cH$, the direct sum of the cohomology groups
of $\Hilb_n S$ over all $n$, see
\cite{Ba,CO,Groj,Lic,L2,Nak1} to name a few. 
There is a related story
in $K$-theory which in many cases is based on Haiman's
character theory of the Bridgeland King and Reid isomorphism 
which identifies $K(\Hilb_n \C^2)$ as an inner-product space
with the ring of symmetric polynomials in infinitely many variables
\cite{BKR,CNO,Ne,SV}.
In some cases, the resulting character theory
leads to functional properties of the generating function 
of cohomological or $K$-theoretic constants in a variable $q$, 
over the number of points $n$ \cite{C,CO,KZ,ZHU}. 

These phenomena are closely related to a physical conjecture known
as AGT (Alday, Gaiotto, Tachikawa) \cite{AGT}, which connects 
correlation functions in four dimensional gauge theory 
with a certain Liouville theory.
In fact, there are two current mathematical proofs of this conjecture 
that proceed along these lines \cite{MO2,V2}.
It would be very desirable mathematically and physically to discover
integrals on a larger moduli space which restrict to both sides of
this dictionary under different specializations of the 
equivariant parameters. The motivation in extending the projection
formula is that interesting integrands on a Grassmannian manifold
are simply easier to construct than interesting moduli spaces.
Haiman's theory makes sense when the moduli space
is the Hilbert scheme, whereas the structures on 
the cohomology and $K$-theory of $\cM_{r,n}$ also lead to interesting
character theory. A fundamental example is the action of 
the Ka\c{c}-Moody algebra $\widehat{sl_r\C}$ on $H^*(\cM_{r,n})$ 
\cite{Lic,Nak6,Ne,NO}, which prompted the Ka\c{c}-Moody example.

\emph{Acknowledgements.}
The author would like to thank the Simons foundation, for its support,
as well as Hiraku Nakajima, Alexei Oblomnkov, and Andrei Okounkov, 
for many valuable discussions.

\section{Plethysm}

\label{contour}

Let $K_T(X)$ denote the complex equivariant $K$-theory of a smooth
complex projective variety with an action of a torus $T= (\C^*)^d$.
Let $\lambda^i$ denote the usual operation 
defined on bundles by
\[\lambda^j([\cE]) = \left[\Lambda^j(\cE)\right].\]
The total operation is defined by
\begin{equation}
\label{virtuallambda}
\lambda(\cE) = \lim_{w \rightarrow 1} \lambda^j(w\cE)=
\lim_{w \rightarrow 1} \sum_j (-1)^j w^j \lambda^j(\cE)
\end{equation}
where the limit is the analytic continuation to $w=1$ of the rational
function defined by the right hand side for $w$ near zero.
The limit exists if $\gamma=[\cE-\cF]$ for honest bundles $\cE,\cF$,
with $\lambda(\cF)$ invertible, and equals
$\lambda(\cE)\lambda(\cF)^{-1}$.

If $\gamma = \sum_I a_I x^I$ for $a_I \in \Z$, and $x^I$ are monomials in some
set of indeterminants, called \emph{plethystic}
variables, we define
\[\lambda\left(\gamma\right)=
\prod_{I} (1-x^I)^{a_I},\quad \gamma^* = \sum_I x^{-I},\]
\begin{equation}
\label{lambda}
\dim(\gamma) = \sum_I a_i,\quad
\det(\gamma) = \prod_I x^{a_I I}.
\end{equation}
Given a symmetric polynomial $f\in \Lambda_n$, we also have a
homomorphism defined in the elementary symmetric function basis by
\[\gamma \mapsto f(\gamma),\quad 
e_{i_1}\cdots e_{i_k}(\gamma) = 
\la^{i_1}(\gamma) \cdots \la^{i_k}(\gamma).\]
We may think of $\gamma$ as an element of $K_T(pt)$, when the
plethystic variables are the torus variables $z_i$.
In this paper, every variable will be considered plethystic,
meaning it counts as an indeterminant for the purposes of \eqref{lambda}.
Furthermore, we will often identify a torus representation
and its character in $\Z_{\geq 0}[z_i^{\pm 1}]$, denoting both by the same letter.
%
%
%
%
%
%
%
%

\section{The Grassmannian}

Suppose $Z$ is a representation of $T$ of dimension $d$, 
and consider the Grassmannian variety of
codimension $n$ subspaces of $Z$,
\[X= G_{n,Z} = \left\{ V \subset Z\big| \codim(V) = n\right\}.\]
There is a tautological bundle $\cV$ 
whose fiber over $V \subset Z$ 
is $V$ itself, and a rank $n$ quotient bundle $\cU = \cZ/\cV$, where
$\cZ = G_{n,Z} \times Z$.
The action of $T$ on $Z$ induces an action on the Grassmannian, and on
the above bundles.
We may consider the characters of the Cech cohomology groups
\[\chi^i (\cE) = \ch H^i_X(\cE),\quad \chi = \sum_i (-1)^i \chi^i.\]
%
%
%
%
%
Only $\chi$, however, descends to a map on $K$-theory. 

Let $P=G_{1,Z}$, and let us define a linear map
\[\xi^\eps_x : 
x^m \mapsto \chi_P^\eps(\cU^m)\in \C[z_i^{\pm 1}],\quad \eps=\emptyset,0,1,...\]
where $\chi^\emptyset=\chi$. The answer is well known to be
%
\[\xi_x^0 f(x) = [x^0]\fl_{x^{-1}} f(x)\cX_x,\quad
\xi_x^d f(x) = [x^0] \fl_{x} f(x)\cX_x,\]
\begin{equation}
\label{xieps}
\cX_x=\la(Zx^{-1})^{-1},\quad
\xi_x^i f(x)=0, \quad i \notin \{0,d\}.
\end{equation}
where $[x^i]$ denotes the coefficient of $x^i$, and
\[\fl_{x^{\pm 1}} : \somefield(x) \hookrightarrow \somefield((x^{\pm 1}))
\subset \somefield[[x^{\pm 1}]]\]
is the map that sends a rational function over $\somefield$
to its Laurent series about $x=0$ or $\infty$ respectively.

Equivariant localization gives a second expression for the Euler
characteristic,
\[\xi_x f(x) = \sum_{j} \xi_{x,j} f(x),\quad
\xi_{x,j}(f(x)) = \chi\left(\la(\cN^*_{P/F_j})^{-1}\iota_{x,j}f(x)\right),\]
\begin{equation}
\label{iota}
\iota_{x,j} : x^m \mapsto i_{F_j}^*(\cU^m) \in K(F_j)\otimes \C(z),
\end{equation}
where $F_j=G_{1,Z_j}$ are the torus fixed components of $P$, 
and $Z_j$ is the invariant subspace of $Z$ with character $z^j$.
We have that
\[K(F_j) \cong \C[y]/(y^c),\quad c = \dim F_j.\]
The restriction map is given by
\begin{equation}
\label{iFa}
\iota_{x,j}(x^m)
= z^{jm}\fl_y(1+y)^m \in \C[[y]] \rightarrow \C[y]/(y^c).
\end{equation}
The pushforward map is
\begin{equation}
\label{chiFa}
\chi: K(F_j)\rightarrow \C,\quad
y^i \mapsto {c-1 \choose i}.
\end{equation}

There is an elegant expression for the Euler characteristic in terms
of residues, which follows easily from the formulas above:
%
\[\xi_x^{0}(f(x)) = -\Res_{x=\infty} g(x),\quad 
\xi_x^{d}(f(x)) = -(-1)^d \Res_{x=0} g(x),\]
\begin{equation}
\label{res}
\xi_{x,j}(f(x)) = \Res_{x=z^j} g(x),\quad
g(x)=x^{-1}f(x)\cX_x,
\end{equation}
where
\[\Res_{x=c} f(x)= [x^{-1}] \fl_x f(x+c)\]
is the algebraic residue operation.
Then
\begin{equation}
\label{reschi}
\xi_{x}^0+(-1)^d\xi_{x}^d = \xi_{x}=\sum_j \xi_{x,j}
\end{equation}
corresponds to the fact that the sum of the residues of a meromorphic
function on $\C\Pp$ equals zero.

There is a well known formula for the general case to 
in terms of the case $n=1$:
%
\begin{equation}
\label{martin}
\chi^\eps_X(f(\cU)) = \frac{1}{n!} 
\xi^{\eps}_{x_1}\cdots\xi^{\eps}_{x_n} f(x_1,...,x_n)\Delta_x,
\end{equation}
\[\Delta_x = \lambda\left(\sum_{i\neq j} x_ix_j^{-1}\right),\quad
\eps\in \{\emptyset,0\}.\]
The $\eps=\emptyset$ case is a simple instance 
of a theorem of Shaun Martin \cite{Mar} for
general symplectic quotients, see also the Jeffrey-Kirwan 
residue formula \cite{JK}.
The $\eps=0$ case is an application of Borel-Weil theory, we refer
the reader to \cite{EF}.

\section{The fundamental class}
Suppose now that $T$ is one-dimensional with torus parameter $z$,
and define a finite-dimensional torus representation
by its character
\[Z = \sum_{i}d_iz^i \in \Z_{\geq 0}[z^{\pm 1}].\] 
Let $k,k'$ respectively denote the largest and smallest $i$ 
such that $d_i \neq 0$, and let $d=\dim(Z)$.
Fix a positive integer $n$, and let $X = G_{n,Z}$. 

Now define an element $\cX =\cX_{\cU} \in K_T(X)[[w]]$ by
%
%
%
%
\begin{equation}
\label{cX}
\cX_{\gamma} = \sum_i (-w)^i \la^i(\cE_\gamma),\quad
\cE_{\gamma} = A\gamma+B\gamma^*+C\gamma\gamma^*,
\end{equation}
for some Laurent polynomials 
\[A=\sum_i a_iz^i,\quad B=\sum_i b_i z^i,\quad C = \sum_i c_iz^i\]
with integer coefficients.
Then we have that $\cX = \fl_w \cY_{w}$ for some class
\[\cY_{w} \in K_T(X)\otimes \C(z,w),\]
%
which can easily be seen using localization
\[K_T(X) \otimes \C(z,w) \cong 
\bigoplus_{F\subset X^T} K(F)\otimes \C(z,w).\]
If this class is well defined at $w=1$, we define
\[\cY = \cY_1 \in K_T(X) \otimes \C(z).\]

Let us define
\[\chi_\cY(\gamma) = \chi_X \left(\gamma\cY\right) \in \C(z),\]
which should be thought of as the Euler characteristic of $\gamma$
over a subvariety $Y\subset X$ which is the intersection of
a section of $\cE^*$ with the zero section, even though $\cE$ is not
an honest bundle, and such a variety may not exist.
Now suppose that $\gamma=\gamma_{\cU,m}$, where
\begin{equation}
\label{gamma}
\gamma_{A,m}=\det(A)^mf(A),\quad 
f \in \Lambda,\quad m \in \Z.
\end{equation}
It follows easily from \eqref{res} and \eqref{martin}
that there are rational functions satisfying
\[f^{\eps}(z,w) \in \C(z,w),\quad
\fl_w f^\eps(z,w) = \chi^{\eps} \left(\gamma \cX\right),\]
when $\epsilon$ is zero or blank. For instance,
$f(z,w)= \chi_{X}(\gamma \cY_{w})$.
Define a rational function whose power series
in $w$ measures
the contribution from the higher cohomology groups,
\[g(z,w)=f(z,w)-f^0(z,w).\]
%
The following lemmas study the expansion of $g(z,w)$ in the
$z$ direction.

\begin{lemma}
\label{singlemma}
Let $x$ be a variable. We have
\[\nu_z (\xi^0_x (x^m\cX_x)) \geq o_{k'},\quad 
\nu_z (\xi^d_x (x^m\cX_x)) \geq o_k,\quad 
\nu_z \left(\xi_{x,i} (x^m\cX_x)\right) \geq o_i,\]
%
%
\[o_i=mi+\sum_{j \leq -i} a_j(i+j)+
\sum_{j \leq i} b_j(j-i)-\sum_{j \leq i} d_j(j-i).\]

\end{lemma}


\begin{proof}

Let us prove the first bound. Consider the power series
\[\fl_x\left(\cA\right) \big|_{x=z^{-k} x}=
\sum_i f_i(z,w) x^i,\quad
\cA =\cX_x \la(Zx^{-1})^{-1}.\]
Using \eqref{res}, it suffices to show that
\[\nu_z (f_i(z,w)) \geq o_k.\]
To prove this we simply study each factor in $\cA$ separately and use
\[\nu_{z}(fg) \geq \nu_z(f)+\nu_z(g).\]

The others are similar.

\end{proof}


\begin{lemma}
\label{reslemma}
Suppose condition \ref{rescond} of theorem \ref{mainthm} below 
is satisfied. Then
\[g(z,w) = \sum_{r\geq 1} (-1)^{rd}{n \choose r} f^{r,n-r}(z,w),\]
where
\[f^{r,s}(z,w) = \sum_{j_1,...,j_s} \Res_{\{y_q=z^{j_q}\}}
\Res_{\{x_{p}=0\}} \Omega_{r,s},\]
\[\Omega_{r,s} = \Omega_{x_1+\cdots+x_r+y_1+\cdots+y_s},\quad
\Res_{\{e_1,...,e_k\}}=\Res_{e_1}\cdots\Res_{e_k},\]
\[\Omega_{A} = \frac{1}{n!}\gamma_{A}\la\left(w\cE_{A}\right)
\la(ZA^*)^{-1}\Delta_A.\]
%
\end{lemma}

\begin{proof}
Using \eqref{res}, \eqref{reschi}, \eqref{martin}, we may write
\[\fl_w g(z,w) = \sum_{r\geq 1} (-1)^{rd}{n \choose r} 
\fl_w \tilde{f}^{r,n-r}(z,w),\]
where $\tilde{f}^{r,s}(z,w)$ is defined by
\[\fl_w \tilde{f}^{r,s}(z,w)=\xi^d_{x_1}\cdots\xi^d_{x_r}
\xi_{y_{1}}\cdots\xi_{y_s}\cX_{x_1+\cdots+x_r+y_1+\dots+y_s}=\]
%
%
%
%
\[\sum_{j_1,...,j_{s}}\Res_{\{y_q=z^{j_q}\}}
\Res_{\{x_p=0\}} \fl_w \Omega_{r,s}.\]
It suffices to prove that
\begin{equation}
\label{pxycom}
\left(\Res_{\{y_q=z^{j_q}\}}\Res_{\{x_p=0\}} \fl_w-
\fl_w \Res_{\{y_q=z^{j_q}\}}\Res_{\{x_p=0\}}\right)\Omega_{r,s}=0,
\end{equation}
so that $f^{r,s}(z,w)=\tilde{f}^{r,s}(z,w)$.

We now claim the following commutation relations between $\fl_w$
and the residue 
%
\begin{equation}
\label{pxcom}
\left(\Res_{x_p=0} \fl_w-\fl_w \Res_{x_p=0}\right)\Omega_{r,s}=
\fl_w \sum_{q,i} \Res_{\{x_p=wz^{i}y_{q}\}}\Omega_{r,s},
\end{equation}
\begin{equation}
\label{pycom}
\left(\Res_{\{y_q=z^{j_q}\}} \fl_w-\fl_w \Res_{\{y_q=z^{j_q}\}}\right)
\Res_{\{x_p=0\}}\Omega_{r,s}=0.
\end{equation}
Both are algebraic facts which may be 
described in terms of formal distributions,
but it is simpler to imagine the residues about
zero as contours about $|x_p|=\epsilon$.
The first commutator may be thought of as the residues
picked up from from swapping the range
$|w|\ll|x_p|$ for $|x_p|\ll|w|$.
The second is also straightforward.


We have that $\Res_{x_p=wz^{i}y_{q}} \Omega_{r,s}$
vanishes to order
\[a_{-i-j}+b_{i+j}+c_{i}-d_{j}+1\]
at $y_{q}=z^j$.
By condition \ref{rescond}, this number is nonnegative, so that
\begin{equation}
\label{xvan}
\Res_{\{y_q=z^{j_q}\}} \Res_{x_p=wz^{i}y_{q}} \Omega_{r,s}=0.
\end{equation}
Furthermore, taking additional residues at $x_{i_p}=0$ can only
increase the degree of vanishing.
Applying this to \eqref{pxcom} and combining with \eqref{pycom}
establishes \eqref{pxycom}, proving the lemma.

\end{proof}

\section{The main theorem}
Now suppose that $Z$ is an infinite-dimensional representation,
defined via its character
\[Z = \sum_{i} d_i z^i \in \Z_{\geq 0}((z)).\]
Suppose furthermore that $A,B \in \Z((z))$, again with coefficients
$a_i,b_i$, respectively.
Let $X=G_{n,Z}$ be the ind-Grassmannian of codimension $n$ subspaces of $Z$,
taken as a limit of subspaces
\[X = \bigcup_k X_k,\quad X_k = G_{n,Z_{\leq k}},\]
where $Z_{\leq k}$ is the direct sum of the subspaces with torus
weight $j \leq k$. 

We set
\[\tilde{K}_T(X) = \lim_{\leftarrow} \tilde{K}_T(X_k),\quad
\tilde{K}_T(X_k) = K_T(X_k) \otimes \C[[z]].\]
The above inverse system is determined by the pullback maps $i^*_{ab}$ where
\[i_{ab} : X_a \rightarrow X_b,\quad
V \mapsto \pi^{-1}_{ba}(V),\quad a \leq b,\]
and $\pi_{ab} : Z_{\leq b} \rightarrow Z_{\leq a}$ is the projection map. 
Given an element $\gamma \in \tilde{K}_T(X)$, with components $\gamma_k$,
let us define
\[\chi_{X}(\gamma) = \lim_{k\rightarrow \infty} \chi_{X_k}(\gamma_k)
\in \C((z)),\]
when that limit exists. One way to guarantee existence is to
have that $i^*_F\gamma=0$ for all but finitely many fixed components
$F\in X^T$. If this is the case, we will say that $\chi_X(\gamma)$
is well defined.

If it exists, we define an element $\cY \in \tilde{K}_T(X)$ by
\[\cY_{k} = \lim_{l\rightarrow \infty} \fl_z \cY_{k,l,1} \in \tilde{K}_T(X_k),\]
where $\cY_{k,l,w}$ is the class obtained by replacing $A,B$
by the finite dimensional spaces $A_{\leq l},B_{\leq l}$ in the definition
of $\cY_w$.

%
%

\begin{thm}
\label{mainthm}
Suppose $\gamma=\gamma_{\cU,m}$,
and the following conditions are satisfied:

\begin{enumerate}[a)]
\item $a_i\geq 0$ for $i\leq -k'$.
\label{acond}
\item \label{bcond} $b_i\geq 0$ and $b_i=d_i$ for large enough $i$.
\item \label{ccond} $c_i=0$ for $i\leq 0$, and $c_i=0$ for large enough $i$.
\item \label{rescond}
For any weights $i,j \in \Z$ with $d_j\neq 0$, $c_i<0$,
we have
\[a_{-i-j}+ b_{i+j}+c_i-d_j+1\geq 0.\]
%
\item \label{defcond} Both $\cY$ and $\chi_{\cY}$ are defined
in the sense described above.
\end{enumerate}
Then for large enough $m$ we have the projection formula,
\begin{equation}
\label{projeq}
\chi_{\cY} \left(\gamma\right) = 
\sum_{j\geq 0}(-1)^j\chi_{X} \left(\gamma\lambda^j (\cE)\right).
\end{equation}
\end{thm}

\begin{proof}
Since $f(\cU)$ is arbitrary,
we may assume without loss of generality that $A \in \Z_{\geq 0}[z^{\pm 1}]$.
Let $f^\eps_{k,l}(z,w)$ and $g_{k,l}(z,w)$ denote the rational functions
from the last section with $Z_{\leq k},B_{\leq l}$ in place of $Z,B$,
and let
\[f^\eps_{k}(z,w) = \lim_{l\rightarrow \infty} 
f^\eps_{k,l}(z,w),\]
%
pointwise. By condition \ref{defcond}, $f_k(z,w)$ is defined at $w=1$.
By definition, the limit over $k$ of $\fl_z f_k(z,1)$ 
agrees with the left hand side of \eqref{projeq}. 

We next claim that the limit of $f_k^0(z,1)$ 
agrees with the right hand side of \eqref{projeq}.
Since the higher cohomology of $f(\cU)$
vanishes for large enough dimension of $Z_{\leq k}$, we have
%
\begin{equation}
\label{ff0}
\lim_{k \rightarrow \infty} \fl_z\fl_wf_k(z,w)=
\lim_{k\rightarrow \infty} \fl_z\fl_w f_k^0(z,w).
\end{equation}
Using condition \ref{ccond}, lemma \ref{singlemma}, 
and \eqref{martin}, we can see that 
\begin{equation}
\label{f0bound}
\lim_{i\rightarrow \infty} \nu_z(e_{k,i}(z)) = \infty,\quad
\fl_w f_k^0(z,w) = \sum_i e_{k,i}(z) w^i.
\end{equation}
uniformly in $k$. This shows that the left side of \eqref{ff0}
converges in $\C((z))$ at $w=1$. 
It also shows that
\[(\fl_z\fl_w-\fl_w\fl_z)f_k^0(z,w) \in \C((w))((z)),\]
whereas a priori, it is only an element of $\C[[z^{\pm 1},w^{\pm 1}]]$.
On the other hand, by multiplying by the denominator of $f_k^0(z,w)$, 
we see that it is a zero divisor,
and so must be zero. But the left side of \eqref{ff0}
agrees with the right side of \eqref{projeq} at $w=1$,
proving the claim.

We now prove that they are equal.
By lemma \ref{singlemma}, we find that
\[\nu_z\left(f_{k,l}^{r,s}(z,w)\right) \geq rmk+ak+b\]
for some consants $a,b$. The constants may be chosen independently
of $l$, by condition \ref{bcond},
and because $o_k$ depends only on the differences $b_i-d_i$
for $i\leq k$. By lemma \ref{reslemma}, $g_{k,l}(z,w)$ 
is a linear combination of $f^{r,s}_{k,l}(z,w)$ for $r\neq 0$, so that
\[\lim_{k\rightarrow \infty} \fl_zf_k(z,w)-
\fl_z f_k^0(z,w)=
\lim_{k \rightarrow \infty} \lim_{l \rightarrow \infty}
\fl_z g_{k,l}(z,w) = 0,\]
including the value $w=1$, as long as $m>-a$.

\end{proof}

\section{Examples}

\subsection{The Hilbert scheme of points in the plane}

\label{hilbsubsec}

Let $Y=\Hilb_n \C^2$, the Hilbert scheme of $n$ points in the plane.
There is a standard torus action on $Y$ induced by 
pullback of ideals from the action on the plane
\begin{equation}
(z_1,z_2) \cdot (x,y) = (z_1^{-1}x,z_2^{-1}y).
\label{hilbtorus}
\end{equation}
%
%
%
The fixed points of $Y$ are the monomial ideals indexed by Young diagrams
%
\[I_\mu = (x^{\mu_1},x^{\mu_2-1}y,...,y^{\ell(\mu)}) \subset R = \C[x,y].\]
The character of the cotangent space to this point 
is a polynomial in $z_i$ with nonnegative integer coefficients
summing to $\dim(Y)=2n$. By deformation theory and 
a standard \c{C}ech cohomology 
argument, it is given by
\begin{equation}
T_\mu Y = \chi (R,R)-\chi(I_\mu,I_\mu),
\label{Thilb}
\end{equation}
where $\chi$ is the Euler characteristic
\[\chi(F,G) = \sum_{i} (-1)^i\ch \Ext^i_{R}(F,G).\]
There is an interesting formula for this polynomial
in terms of the arm and leg lengths of boxes in $\mu$,
which we will not need. See
\cite{CO,Nak2} for a reference on this calculation.

Now let $Z$ be the total space of $R$, so that
\[Z = \fl_{z_1,z_2} M^{-1},\quad M = (1-z_1)(1-z_2).\]
Let $X=G_{n,Z}$, and let
\[A = -z_1z_2,\quad B = Z-1,\quad C = M-1,\]
with $\cE$ and $\cY$ as in the last section. 
There is an injection $Y \hookrightarrow X$ determined by
sending an ideal to its total space in $Z=R$, which
comes up in the construction of $Y$.
The images of the fixed points are
\[V_\mu = H^0(I_\mu) \subset Z,\quad 
U_\mu = Z/V_\mu=\sum_{(i,j)\in \mu} z_1^j z_2^i,\]
where $(i,j)$ are the coordinates of a box in $\mu$.

\begin{lemma}
\label{vanishing}
Suppose $U \in X^T$ is an invariant subspace.
%
\begin{enumerate}[a)] 
\item \label{hilbvanish}
We have that $\lambda\left(\cE_U\right)$
vanishes unless $U=U_\mu$ for some $\mu$.
\item \label{tgrasshilb}
If $U = U_\mu$, then
\[T_\mu^*X-T^*_\mu Y = \cE_\mu,\quad \cE_\mu=\cE_{U_\mu}.\]
\end{enumerate}
\end{lemma}

\begin{proof}

For part \ref{hilbvanish},
it suffices to show that the constant term of
%
$\cE_U$
is positive unless $U=U_\mu$, in which case it is zero.
Consider the graph whose vertices are $\Z^2 \subset \R^2$,
and whose edge set $E$ connects horizontal and vertical neighbors. 
Color each box with lower-left corner $(i,j)$ white
if $z_1^iz_2^j$ is a weight of $V$, and black otherwise.
Define subsets by
\[X_0 = \left\{v \in \Z^{2} :\mbox{$v(\nearrow)$ is black, 
$v(\swarrow)$ is white}\right\}\]
\[X_1 = \left\{e \in E :\mbox{$e(\nearrow)$ is black,
$e(\swarrow)$ is white}\right\}\]
Here $v(\nearrow)$ is the upper-right neighboring box to $v$,
$e(\nearrow)$ is the upper or right neighboring box to the edge
$e$ depending on whether $e$ is horizontal or vertical, and
similarly for the southwest arrow.

Expanding $\cE$, we see that the constant term is
\[[z_1^0z_2^0] \cE_U = x_1-x_0,\quad x_i = |X_i|.\]
Now notice that every vertex in $X_0$ is the endpoint of 
exactly two edges in $X_1$, but each edge in $X_1$ always has at most
two endpoints in $X_0$, proving that $x_1 - x_0 \geq 0$. 
If $U$ does not come from a Young diagram, 
then the set $X_1$ is nonempty, and
there must be some edge in $X_1$ whose endpoints are not 
both in $X_0$, 
leading to strict inequality.

Part \ref{tgrasshilb} may be deduced easily from \eqref{Thilb}, and 
%
\[\chi (I_\mu,I_\nu) = z_1^{-1}z_2^{-1} M V_\mu^*V_\nu.\]
%

\end{proof}

Now restrict to a one-dimensional torus $z_i = z^{a_i}$, where the
$a_i$ are large enough that the fixed points of $Y$ are isolated.
Lemma \ref{vanishing} combined with localization on 
each fixed component $F$ with respect
to the two-dimensional torus proves that condition \ref{defcond} of the
theorem is satisfied. Condition \ref{rescond}
holds because $Z$ has the same character as $R$,
$B$ is the character of the total
space of its maximal ideal $\mathfrak{m}$, and
$xR$ and $yR$ are contained in $\mathfrak{m}$.
The others are obvious.

By lemma \ref{vanishing}, we see that
$\chi_{\Hilb_n}(\gamma_{\cU,m})=\chi_{\cY}(\gamma_{\cU,m})$,
where $\cU$ also denotes the tautological rank $n$ bundle 
on $Y=\Hilb_n$, which is pulled back from $X$.
By theorem \ref{mainthm}, we have 
%
\begin{equation}
\chi_{\Hilb_n} \left(\gamma_{\cU,m} \right) = 
\sum_j (-1)^j \chi_X \left(\gamma\lambda^j\left(\cE\right)\right)
\in \Z[[z]]
\label{hilbeq}
\end{equation}
for large enough $m$.
%
%
Since the answer is a rational functions of $z_i$, it is
determined by its values on the restricted torus. We may therefore
drop the assumption that $z=z^{a_i}$, and have an equality
of functions of two distinct torus variables $z_i$.
It may be checked that both sides are given by elements of
$\C(z_i)[z_i^m]$ for $m\geq 0$, and so \eqref{hilbeq} holds for all
nonnegative $m$. 
The point of this formula is the the right hand side is 
given explicitly by a sum of power series
with integer coefficients.



%
%

\subsection{The Hilbert scheme of a singular curve}

Let $C$ denote the singular curve $y^2 = x^3$, and consider the action
\[T = \C^* \circlearrowright C,\quad z\cdot(x,y) = (z^{-2}x,z^{-3}y).\]
Let $Y$ denote the Hilbert scheme of $n$ 
points in this curve, whose points correspond to ideals in
\[R = \C[x,y]/(y^2-x^3) \cong \C[u^2,u^3],\]
with $\dim_{\C} R/I = n$.
The torus fixed points of $Y$ are those of the form
\[I_S = \bigoplus_{i \in S} \C\cdot u^i \subset R,\]
for $S$ a sub-semigroup of $\{0,2,3,4,...\}$.

There is an injection $Y \hookrightarrow X$, where the
data for $X$ is given by
\[Z=\ch R = \fl_z (1-z^6)M^{-1},\quad 
M = (1-z^2)(1-z^3).\]
Now let
\[A=-z^5,\quad B= Z+z^6-1,\quad C = M-1.\]
We find that $\cY_U$ vanishes at all fixed points in $X$
except those whose weights form a semigroup $S$.
We would find that $T_S^* X-\cE_S$ does
not consist entirely of nonnegative weights, but that
the signed dimension is always $n$. This corresponds to
the fact that the Hilbert scheme of points on this curve
only has a virtual tangent bundle of expected dimension $n$.
Its character may be calculated by realizing $Y$ as an
lci subvariety of the Hilbert scheme of $n$ points in the plane.
For a reference, see \cite{Sh}. 

In a similar way to the last subsection, theorem \ref{mainthm}
gives a power series formula for the Euler characteristic, 
but now for $m\geq 1$.

\subsection{The affine Grassmannian}

Let $G=SL(2,\C)$, and consider the affine Grassmannian 
\[Y=LG_{\C}/L^+G_{\C},\]
 where $LG$ is the space of maps from the circle into 
$G$, and $L^+G$ are those maps which extend to a holomorphic function in
the disc of radius $1$. 

In \cite{S}, Segal noted that there should
be a proof of the Weyl-Ka\c{c} character formula using this variety,
which is analogous of the well-known geometric proof of the Weyl character
formula using $K$-theoretic localization combined with Borel-Weil-Bott,
see \cite{CG}. He also pointed 
out that there was a gap in the reasoning due to the
fact that $Y$ is infinite-dimensional with singular closure, 
and the explanation that the higher cohomology groups vanish. 
This topic, and generalizations to
the related flag varieties have been studied by several authors, including
\cite{Ku,T1,T2}. 

We now demonstrate how theorem \ref{mainthm} 
can be used to circumvent these two difficulties in the case
of the Jacobi triple product formula, 
which corresponds to the basic representation for $G=SL(2,\C)$.
It would be interesting to see how far this approach 
generalizes.
Let us ignore technicalities and simply motivate the
choice of data for theorem \ref{mainthm}.
There is an action of a two-dimensional torus on $Y$ by
\[
(g \cdot f)(x) = \Ad
\left(\begin{array}{cc} z^{-1} & \\ & z \end{array}\right)\cdot
f(q x),\quad g = (q,z).\]
%
Ignoring the infinite-dimensionality of $Y$,
we can write down the character of the cotangent bundle to this space
at a fixed point as follows. 
The cotangent bundle at the image of the identity in $Y$
is given by
%
\begin{equation}
\label{taff}
T^*_1Y = \left(L\mathfrak{g}_{\C} / L^+ \mathfrak{g}_{\C}\right)^*=
\frac{q}{1-q}(z^2+1+z^{-2}).
\end{equation}
The character at a general fixed point can be extracted from by
applying elements of the affine Weyl group.

Let $\cH = L\C^2$, the Hilbert space
of maps to $\C^2$. Then $LG_{\C}$ acts in the obvious way on this space,
and $L^+G_{\C}$ is precisely the subgroup that preserves the subspace
$V \subset \cH$ of all maps which are holomorphic at the origin.
The action on $\cH$ induces an inclusion $Y \subset X$ in which
$1$ maps to $V$,
where $X$ is the Sato Grassmannian of half-infinite dimensional subspaces
of $\cH$, by taking the orbit space of $V$. The character of the cotangent
bundle at $V$ is given by
\begin{equation}
\label{tsato}
\ch T^*_V X = \frac{q}{(1-q)^2} (z^2+2+z^{-2}),
\end{equation}
%
%
%

Now let us derive the Jacobi triple product.
For each $n$, let 
%
%
\[M = 1-q,\quad W=q^{-n}(z+z^{-1}),\quad Z = q^{-n}(z+z^{-1})\fl_zM^{-1},\]
\[A=0,\quad B=Z-W,\quad C=M-1,\quad X^{(n)}=G_{2n,Z},\]
so that $Z\subset \cH$, and includes the 
whole space as $n$ becomes large. As in section \ref{hilbsubsec},
we find that the projection formula holds for the two-dimensional
torus, and that $m\geq 0$ is sufficient.

We may check that $\lambda(\cE^{(n)})$ vanishes at all fixed points of 
$X^{(n)}$ except those whose complementary subspace has character
\[U_k = \sum_{-n \leq i \leq k-1} zq^i+\sum_{-n \leq i \leq -k-1} z^{-1}q^i,\]
and the character at such a point satisfies
\[\lim_{n\rightarrow \infty} 
\left(T_k^*X^{(n)}-\cE^{(n)}_k-T^*_k Y\right)=\frac{q}{1-q}.\]
%

Now taking the limit over $n$ of \eqref{projeq}
%
gives
\[\sum_k (q;q)_\infty^{-1}\theta(z^2,q)^{-1}
\left(z^{4k}q^{2k^2+k}-z^{4k-2}q^{2k^2-k}\right)=\]
\begin{equation}
\label{jtp}
\sum_{j} (-1)^jq^j \lim_{n \rightarrow \infty}\chi_X 
\left(\lambda^j(T^*X)\right)=(q;q)_\infty^{-1},
\end{equation}
where
\[(x;q)_\infty = \prod_{i \geq 0} (1-xq^i),\quad 
\theta(x;q) = (q;q)_\infty (xq;q)_\infty (x^{-1};q)_\infty.\]
The second equality follows from
\[\chi_{\Gr(k,n)}\left(\lambda^j(T^*)\right) = (-1)^jp(j),\]
for sufficiently large $k,n-k$,
where $p(j)$ is the number of partitions of $j$,
and the answer holds equivariantly for any group action on $\C^n$.

%

%
%
%

\end{document}